\pdfoutput=1
\documentclass[1p,fleqn,preprint]{elsarticle}

\usepackage{amsmath,amssymb,amsthm}
\usepackage[mathscr]{eucal}
\usepackage{mathtools}
\usepackage{graphicx}
\usepackage{subfigure}

\usepackage{ifpdf}
\ifpdf
\usepackage[bookmarks=false,
pdfstartview=FitH,linkbordercolor={0.5 1 1},
citebordercolor={0.5 1 0.5},unicode,
hyperindex]{hyperref}
\else
\usepackage{breakurl}
\fi

\biboptions{numbers,sort&compress}

\newtheorem{theorem}{Theorem}
\newtheorem{lemma}{Lemma}
\newtheorem{corollary}{Corollary}
\newtheorem{example}{Example}

\newcommand{\bR}{\mathbb{R}}
\newcommand{\bZ}{\mathbb{Z}}
\newcommand{\setJ}{\mathscr{J}}
\newcommand{\setO}{\mathscr{O}}
\newcommand{\setU}{\mathscr{U}}

\makeatletter
\def\ps@pprintTitle{%
     \let\@oddhead\@empty
     \let\@evenhead\@empty
     \def\@oddfoot{\footnotesize\itshape\hfill\today}%
     \let\@evenfoot\@oddfoot}
\makeatother

\begin{document}

\begin{frontmatter}

\title{Hardy type asymptotics for cosine series in several variables with
decreasing power-like coefficients}

\author{Victor Kozyakin}

\address{Institute for Information Transmission
Problems\\ Russian Academy of Sciences\\ Bolshoj Karetny lane 19, Moscow
127994 GSP-4, Russia}

\ead{kozyakin@iitp.ru}
\ead[url]{http://www.iitp.ru/en/users/46.htm}

\begin{abstract}
The investigation of the asymptotic behavior of trigonometric series near
the origin is a prominent topic in mathematical analysis. For trigonometric
series in one variable, this problem was exhaustively studied by various
authors in a series of publications dating back to the work of G. H. Hardy,
1928.

Trigonometric series in several variables have got less attention. The aim
of the work is to partially fill this gap by finding the asymptotics of
trigonometric series in several variables with the terms, having a form of
`one minus the cosine' up to a decreasing power-like factor:
\[
\sum_{z\in\mathbb{Z}^{d}\setminus\{0\}}\frac{1}{\|z\|^{d+\alpha}}(1-\cos\langle
z,\theta\rangle), \qquad \theta\in\mathbb{R}^{d},
\]
where $\langle\cdot,\cdot\rangle$ is the standard inner product and
$\|\cdot\|$ is the max-norm on $\mathbb{R}^{d}$.

The approach developed in the paper is quite elementary and essentially
algebraic. It does not rely on the classic machinery of the asymptotic
analysis such as slowly varying functions, Tauberian theorems or the Abel
transform. However, in our case, it allows to obtain explicit expressions
for the asymptotics and to extend to the general case $d\ge 1$ classical
results of G. H. Hardy and other authors known for $d=1$.
\end{abstract}


\end{frontmatter}

\section{Introduction}

Given a real number $\alpha>0$, consider the function
\begin{equation}\label{E:defFd}
F_{d}(\theta):=\sum_{z\in\bZ^{d}\setminus\{0\}}\frac{1}{\|z\|^{d+\alpha}}(1-\cos\langle z,\theta\rangle),
\qquad \theta\in\bR^{d}.
\end{equation}
Here $\bZ^{d}$ is the lattice of points from $\mathbb{R}^{d}$ with integer
coordinates, $\langle\cdot,\cdot\rangle$ is the standard inner product and
$\|\cdot\|$ is the max-norm on $\mathbb{R}^{d}$ defined by
\[
\langle x,y\rangle=x_{1}y_{1}+x_{2}y_{2}+\cdots+x_{d}y_{d},\quad
\|x\|=\max\{|x_{1}|,|x_{2}|,\ldots,|x_{d}|\},
\]
where $x=\{x_{1},x_{2},\ldots,x_{d}\}$, $y=\{y_{1},y_{2},\ldots,y_{d}\}$.

The series in \eqref{E:defFd} is uniformly convergent for any $\alpha>0$ and
therefore the function $F_{d}(\theta)$ is non-negative and continuous, and
$F_{d}(0)=0$. We will be interested in study of the asymptotic behavior of
$F_{d}(\theta)$ as $\theta\to 0$. An example of the function $F_{d}(\theta)$
for $d=2$ is plotted in Fig.~\ref{F-cossum2}.

\begin{figure}[!htbp]
\begin{center}
\includegraphics[width=0.5\textwidth]{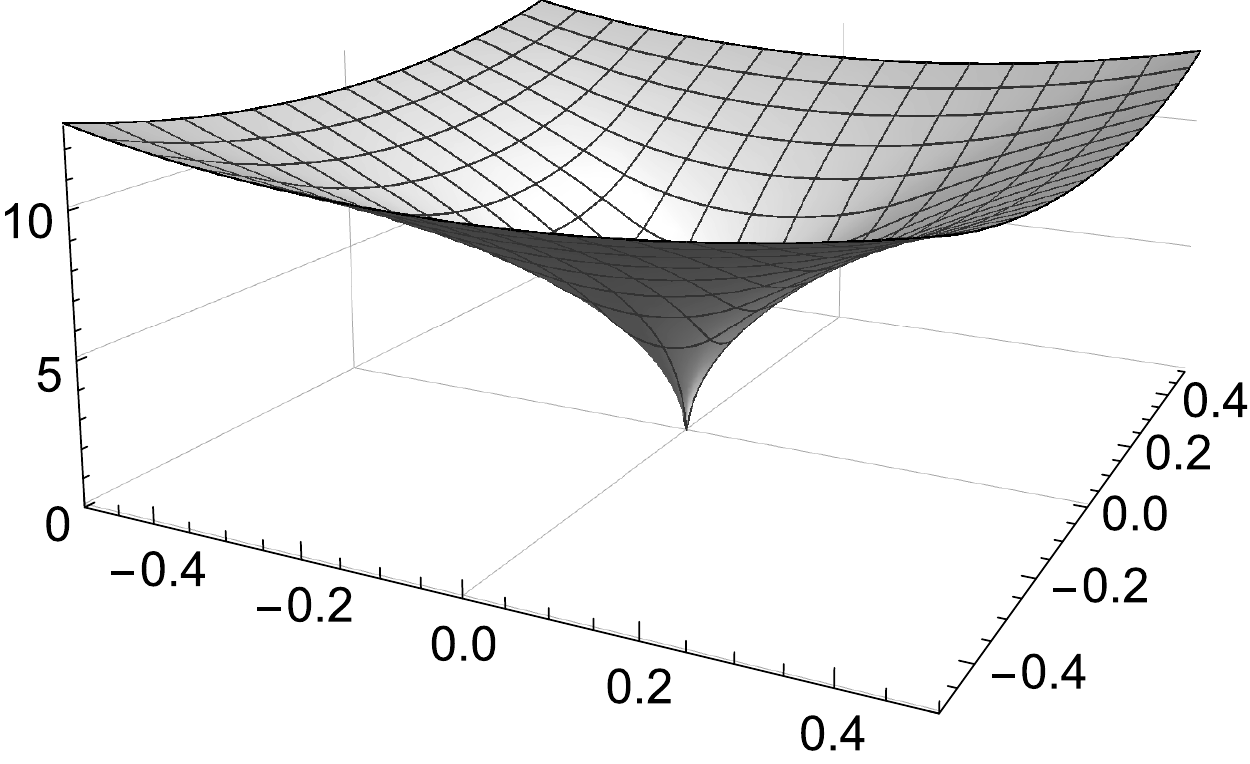}
\caption{Graph of the function $F_{2}(\theta)$ for $\alpha=0.5$.}\label{F-cossum2}
\end{center}
\end{figure}

For $d=1$, the function $F_{d}(\theta)$ can be represented in the form
$F_{d}(\theta)=2H_{\alpha}(\theta)$, where
\begin{equation}\label{E:Chi-def}
H_{\alpha}(\theta):=\sum_{n=1}^{\infty}\frac{1}{n^{1+\alpha}}(1-\cos n\theta),\quad
\theta\in\bR,
\end{equation}
and its asymptotics as $\theta\to 0$ can be described with the help of
classical results going back to the work of G.~H.~Hardy \cite{Hardy:LJMS28}
in which it was shown that for the functions
\[
f(\theta)=\sum_{n=1}^{\infty}a_{n}\cos n\theta,\qquad
g(\theta)=\sum_{n=1}^{\infty}a_{n}\sin n\theta,
\]
where $0<\alpha<1$ and $n^{\alpha}a_{n}\to 1$, the following asymptotics as
$\theta\to 0+$ are valid:
\begin{equation}\label{E:hardy}
f(\theta)\simeq
\Gamma(1-\alpha)\sin\Bigl(\frac{\pi\alpha}{2}\Bigr)\theta^{\alpha-1},\qquad
g(\theta)\simeq
\Gamma(1-\alpha)\cos\Bigl(\frac{\pi\alpha}{2}\Bigr)\theta^{\alpha-1}.
\end{equation}
Here $\Gamma(\cdot)$ is the gamma function, and the notation
$h_{1}(\theta)\simeq h_{2}(\theta)$ as $\theta\to \theta_{0}$, for functions
$h_{1}(\theta)$ and $h_{2}(\theta)$, means that
$h_{1}(\theta)/h_{2}(\theta)\to 1$ as $\theta\to \theta_{0}$.

If $\alpha=1$, that is, $na_{n}\to 1$ then instead of \eqref{E:hardy} there
are valid the following limit relations: $f(\theta)\simeq\log (1/|\theta|)$
and $g(\theta) \to \frac{\pi}{2}$ as $\theta\to 0+$ \cite{Hardy:PLMS31}.
S.~Aljan\v{c}i\'{c}, R.~Bojani\'{c} and M.~Tomi\'{c} proved in
\cite{ABT:PIMA56} that the second of asymptotics \eqref{E:hardy} is valid
also for $0<\alpha<2$, and C.~H.~Yong in \cite{Yong:JMAA71,Yong:JMAA72}
analyzed the asymptotic behavior of the function $f(\theta)$ for all
$\alpha>1$ and of the function $g(\theta)$ for all $\alpha\ge 2$.

So, in \cite{Hardy:LJMS28,Hardy:PLMS31,ABT:PIMA56,Yong:JMAA71,Yong:JMAA72}
the asymptotic behavior of the functions $f(\theta)$ and $g(\theta)$ has been
completely investigated for all $\alpha>0$. In particular, from
\cite{Yong:JMAA71,Yong:JMAA72} it follows that
\begin{equation}\label{E:asympF1}
H_{\alpha}(\theta)\simeq H^{*}_{\alpha}(\theta):=
\begin{dcases}
\frac{1}{\alpha}\Gamma(1-\alpha)\cos\Bigl(\frac{\pi\alpha}{2}\Bigr)|\theta|^{\alpha}&
\text{for~} 0<\alpha<2,\\
\frac{1}{2}\theta^{2}\log\frac{1}{|\theta|}&
\text{for~} \alpha=2,\\
\frac{1}{2}\zeta(\alpha-1)\theta^{2}&
\text{for~} \alpha>2,
\end{dcases}
\end{equation}
where $\zeta(s)=\sum_{n=1}^{\infty}\frac{1}{n^{s}}$ is the Riemann zeta
function. The quantity $\Gamma(1-\alpha)\cos(\frac{\pi\alpha}{2})$ in
\eqref{E:asympF1} is indeterminate when $\alpha=1$ since $\Gamma(0)=\infty$
and $\cos(\frac{\pi}{2})=0$. This indeterminate form can be resolved by
treating $\Gamma(1-\alpha)\cos(\frac{\pi\alpha}{2})$ in the case $\alpha=1$
as
$\lim_{\alpha\to1-}\Gamma(1-\alpha)\cos(\frac{\pi\alpha}{2})=\frac{\pi}{2}$.
This indeterminate form can be resolved also with the help of the identity
$\Gamma(1-\alpha)\cos(\frac{\pi\alpha}{2})\equiv
\frac{\pi}{2\Gamma(\alpha)\sin(\frac{\pi\alpha}{2})}$.

Note that in works
\cite{Hardy:LJMS28,Hardy:PLMS31,ABT:PIMA56,Yong:JMAA71,Yong:JMAA72}, as well
as in subsequent publications
\cite{HR:QJM45,Nurcombe:JMAA93,ChenChen:JMAA00,ChenChen:JMAA00-II,Tikhonov:JMAA07},
one can find quite a number of deeper results than those mentioned earlier,
part of which are included in monograph \cite[Ch.~V]{Zygmund02}.

From \eqref{E:asympF1} it follows that $F_{1}(\theta)\simeq
2H^{*}_{\alpha}(\theta)$ for each $\alpha>0$. Behavior of the functions
$F_{1}(\theta)$ and $2H^{*}_{\alpha}(\theta)$ is illustrated in
Fig.~\ref{F:FdH}.
\begin{figure}[htbp!]
\centering
\hfill\subfigure[$\alpha=0.5$]{\includegraphics*[width=0.3\textwidth]{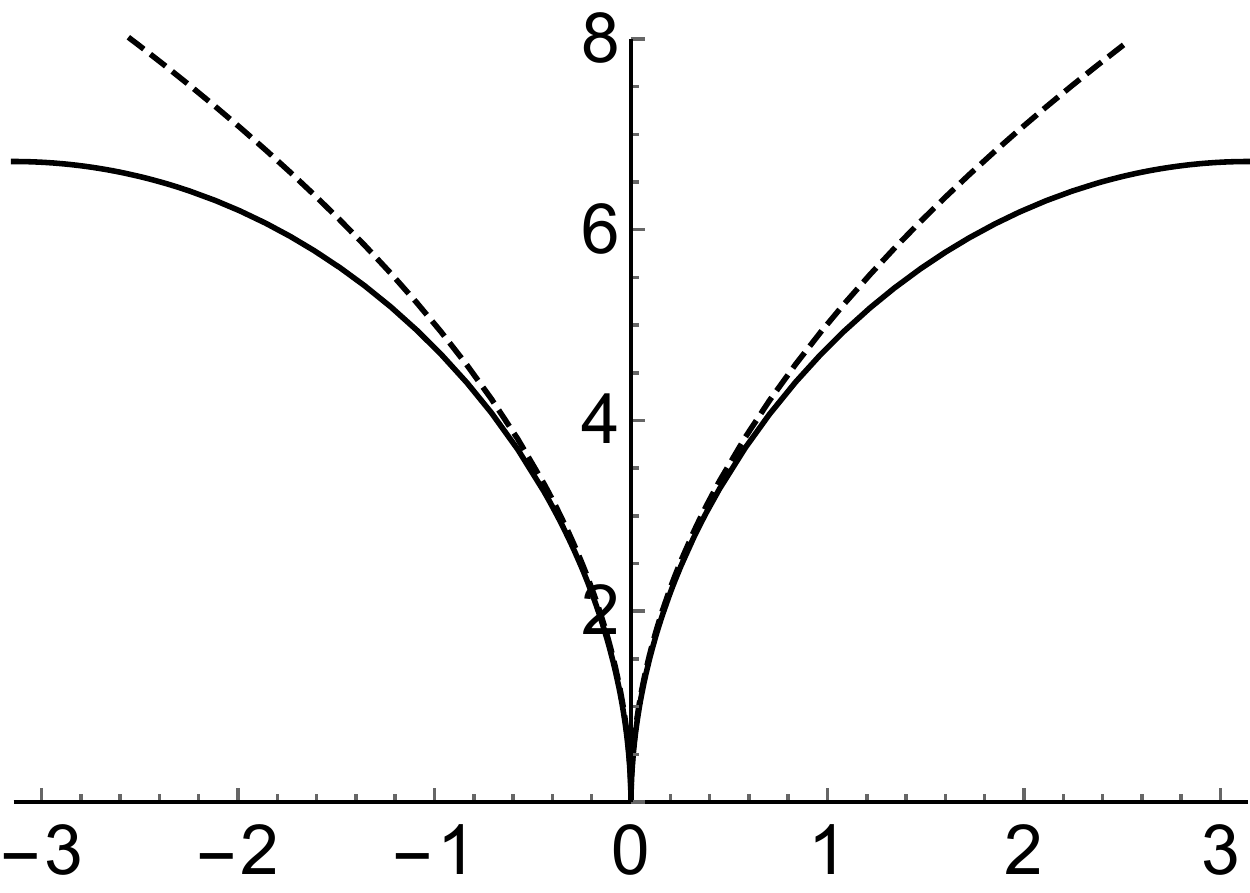}}
\hfill\subfigure[$\alpha=1.0$]{\includegraphics*[width=0.3\textwidth]{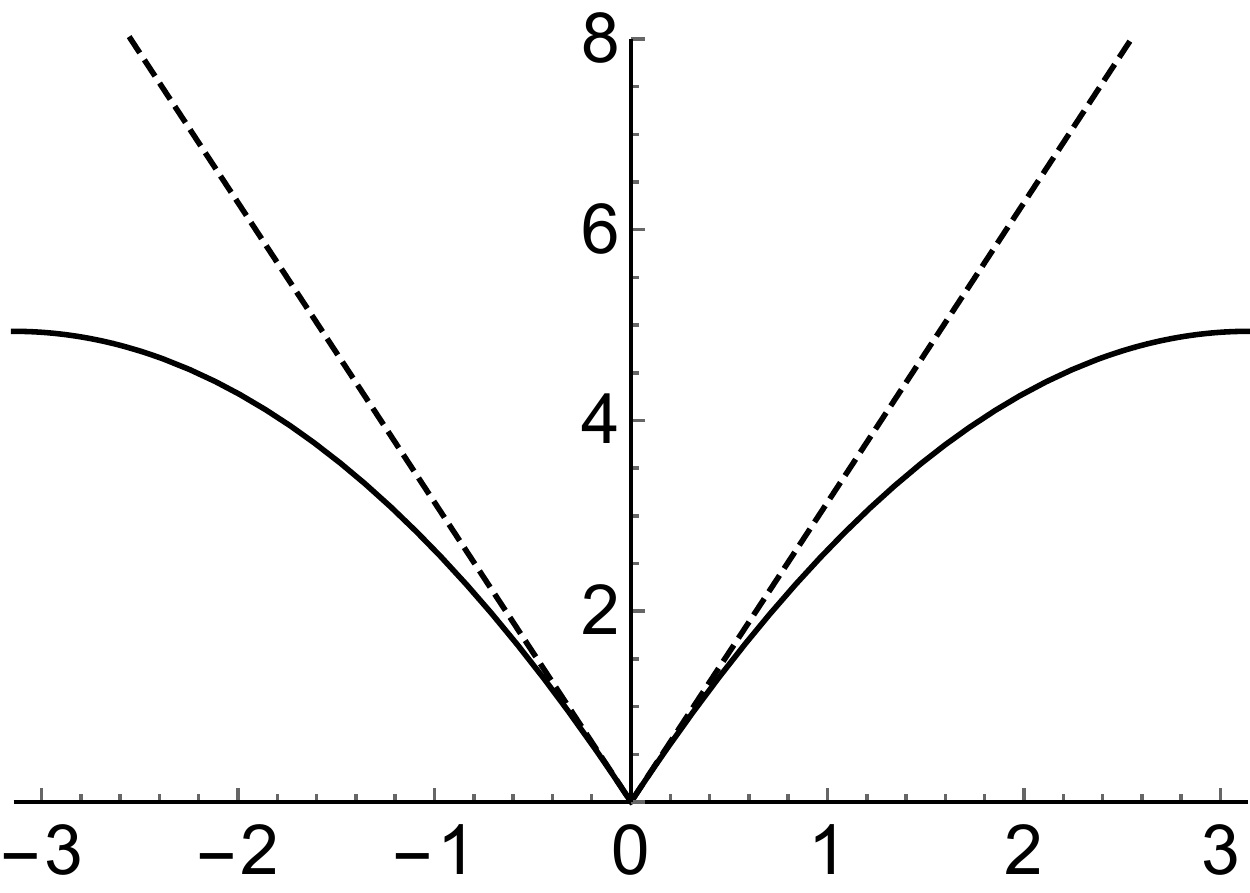}}
\hfill\subfigure[$\alpha=1.5$]{\includegraphics*[width=0.3\textwidth]{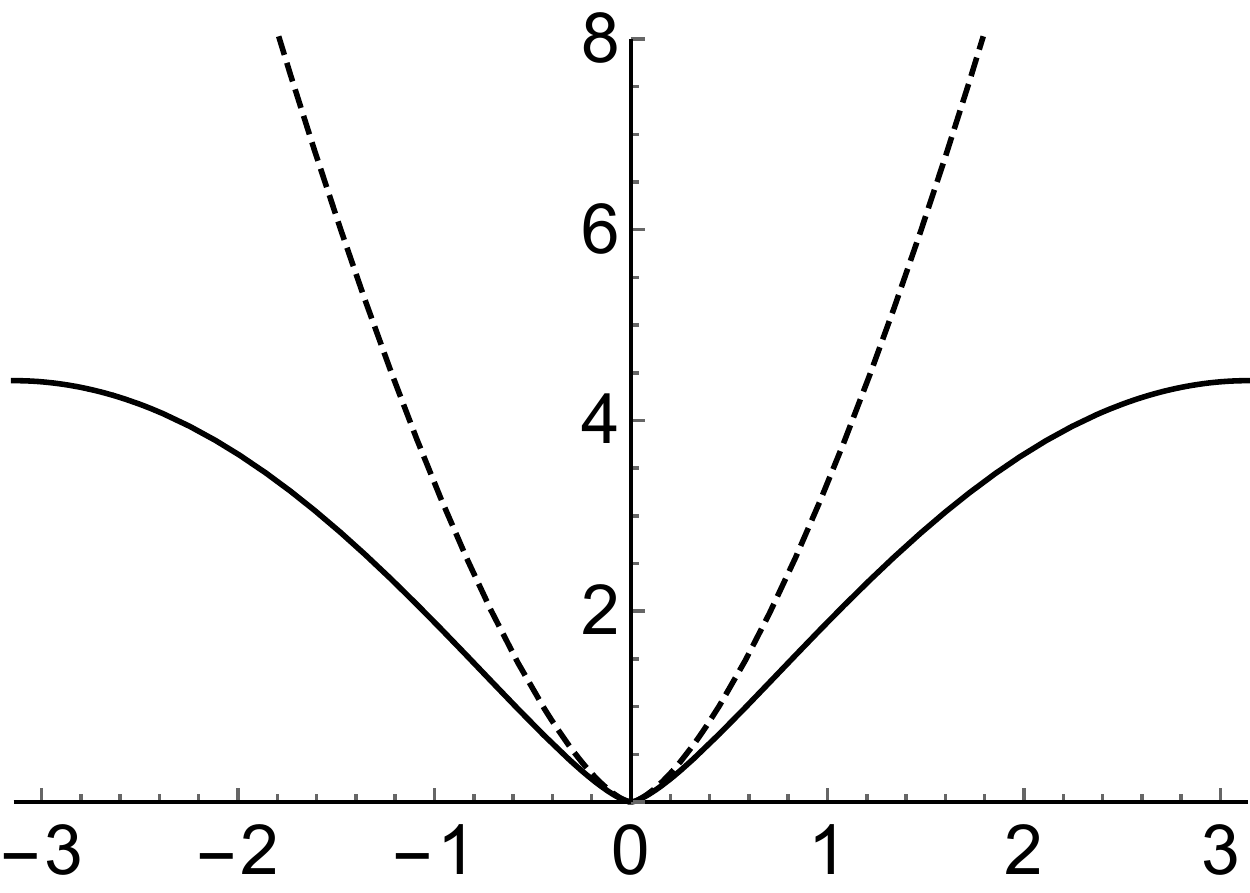}}
\hfill\mbox{}
\caption{Graphs of $F_{1}(\theta)$ (solid line) and
$2H^{*}_{\alpha}(\theta)$ (dash line) for some values of the parameter $\alpha$.}\label{F:FdH}
\end{figure}

For $d\ge 2$ the asymptotic behavior of the function $F_{d}(\theta)$ at zero
is much less studied. We would like to mention that for the function
$F_{2}(\theta)$ E.~Yarovaya \cite{Yar:CSTM13,Yar:NPSMDA14} obtained a lower
bound of order $\|\theta\|^{\alpha}$ near the origin. A similar result for
$F_{3}(\theta)$ has been established by A.~Rytova, a student of E.~Yarovaya
(unpublished).

In connection with this the aim of the work is to prove in
Section~\ref{S:mainresults} Theorems~\ref{T:1}--\ref{T:genas} which provide
an explicit description of the asymptotic behavior of the function
$F_{d}(\theta)$ as $\theta\to0$ for all $d\ge2$.

It is not clear whether it is possible to apply, in the case $d\ge2$, the
methods developed in
\cite{Hardy:LJMS28,Hardy:PLMS31,ABT:PIMA56,Yong:JMAA71,Yong:JMAA72,HR:QJM45,Nurcombe:JMAA93,ChenChen:JMAA00,ChenChen:JMAA00-II,Tikhonov:JMAA07}
for studying the asymptotics of trigonometric series due to their
one-dimensional specificity. Therefore in our case we use the approach of
``reduction to dimension one'' allowing to express the function
$F_{d}(\theta)$ as an explicit combination of the one-dimensional functions
$F_{1}(\cdot)$, or rather of the functions $H_{\alpha}(\cdot)$. This gives an
opportunity to reduce the analysis of the case $d\ge2$ to the case $d=1$. A
similar idea was used in \cite{Yar:CSTM13,Yar:NPSMDA14}.

Outline the structure of the work. In this section we have explained the
motivation for the analysis of asymptotic behavior of the series
\eqref{E:defFd} in several variables, and have presented a concise review of
the publications related to this problem. In Section~\ref{S:mainresults} the
main results are formulated. Theorem~\ref{T:1} there, about representation of
the function $F_{d}(\theta)$ in the form of a finite integral combination of
the functions $H_{\alpha}(\theta)$ with different values of the parameter
$\alpha$, plays the key role. This theorem reduces the analysis of the
behavior of the function $F_{d}(\theta)$ for $d\ge2$ to the case $d=1$. With
the help of Theorem~\ref{T:1}, in Theorems~\ref{T:2} and \ref{T:3} an
explicit form of the asymptotics for $F_{d}(\theta)$ is obtained for all
$\alpha>0$, and in Theorem~\ref{T:genas} the asymptotics of the series $\sum
a_{z}(1-\cos\langle z,\theta\rangle)$ with the coefficients $a_{z}$
asymptotically equivalent to $\|z\|^{-(d+\alpha)}$ is established. The
section is finalized by computation of the asymptotics of the function
$F_{d}(\theta)$ in the cases $d=2,3$. The subsequent
Sections~\ref{S:proofT1}--\ref{S-proofG} are devoted to the proofs of
Theorems~\ref{T:1}--\ref{T:genas}.

\section{Main Results}\label{S:mainresults}

For each $m=1,2\ldots,d$ let $\setJ_{m}$ denote the set of all sequences
$\{j_{1},\ldots,j_{d}\}$ with pairwise distinct elements
$j_{k}\in\{1,2,\ldots,d\}$ such that $j_{1}\le\ldots\le j_{m}$ and
$j_{m+1}\le\ldots\le j_{d}$. In the case $m=d$ the  sequence
$\{j_{m+1},\ldots, j_{d}\}$ is assumed to be empty. Let also $\setU_{m}$
denote the set of all sequences $\{s_{1},\ldots,s_{m}\}$ with the elements
$s_{i}=\pm1$, and $\setO_{d}$ stand for the set of all natural odd numbers
not exceeding $d$.

\begin{theorem}\label{T:1} Let $d\ge2$ and $\alpha>0$. Then
\begin{align}\nonumber
F_{d}(\theta)&=
2\sum_{m\in\setO_{d}}\sum_{\{j_{k}\}\in\setJ_{m}}\zeta(m+\alpha)\Bigl(2^{d-m}-
\Bigl(\theta_{j_{m+1}}\cot\frac{\theta_{j_{m+1}}}{2}\Bigr)
\cdots\Bigl(\theta_{j_{d}}\cot\frac{\theta_{j_{d}}}{2}\Bigr)\Bigr)
\\
\nonumber&+\frac{1}{2^{d-1}}\sum_{m\in\setO_{d}}\sum_{\{j_{k}\}\in\setJ_{m}}
\biggl(\cot\frac{\theta_{j_{m+1}}}{2}\cdots\cot\frac{\theta_{j_{d}}}{2}\\
\label{E:Fd-fin}&\begin{multlined}[b][0.8\linewidth]\quad\times\sum_{\{s_{k}\}\in\setU_{m}}\int\limits_{-\theta_{j_{d}}}^{\theta_{j_{d}}}\cdots
\int\limits_{\mathclap{-\theta_{j_{m+1}}}}^{\mathclap{\theta_{j_{m+1}}}}
H_{m+\alpha-1}(s_{1}\theta_{j_{1}}+\cdots+
s_{m}\theta_{j_{m}}\\+ \eta_{m+1}+\cdots+ \eta_{d})\,d \eta_{m+1}\ldots d \eta_{d}\biggr),
\end{multlined}
\end{align}
where $\zeta(s)=\sum_{n=1}^{\infty}\frac{1}{n^{s}}$ is the Riemann zeta
function.
\end{theorem}

By using asymptotic equalities \eqref{E:asympF1} for the function
$H_{\alpha}(\cdot)$, from the previous theorem it is easy to derive an
explicit form for the asymptotics of $F_{d}(\theta)$, for different values of
the parameter $\alpha$. Define
\begin{equation}\label{E:defAd}
A_{d}(\theta):=\sum_{\{j_{k}\}\in\setJ_{1}}\frac{1}{\theta_{j_{2}}\cdots\theta_{j_{d}}}
\int\limits_{-\theta_{j_{d}}}^{\theta_{j_{d}}}\cdots
\int\limits_{\mathclap{-\theta_{j_{2}}}}^{\mathclap{\theta_{j_{2}}}}
|\theta_{j_{1}}+ \eta_{2}+\cdots+ \eta_{d}|^{\alpha}\,d \eta_{2}\ldots d \eta_{d}.
\end{equation}
The set $\setJ_{1}$, over which summation in \eqref{E:defAd} is made,
consists of $d$ sequences $\{j_{k}\}$. These sequences are formed as follows:
the element $j_{1}$ takes all the values $1,2,\ldots, d$ and, for a chosen
$j_{1}$, the sequence of values $\{j_{2},j_{3},\ldots,j_{d}\}$ is obtained
from $\{1,2,\ldots,d\}$ by removing the member $j_{1}$.

\begin{theorem}\label{T:2}
Let $d\ge2$ and $\alpha>0$. Then
\begin{equation}\label{E:asympA}
F_{d}(\theta)\simeq\begin{dcases}
\vphantom{\Bigg|}\frac{2}{\alpha}\Gamma(1-\alpha)\cos\Bigl(\frac{\pi\alpha}{2}\Bigr)
A_{d}(\theta)&\text{for~} \alpha<2,\\
\vphantom{\Bigg|}\frac{2^{d-1}(d+2)}{3} |\theta|^{2}\log\frac{1}{|\theta|}
&\text{for~} \alpha=2,\\
\vphantom{\Bigg|}\frac{1}{6}\biggl(\sum_{n=1}^{\infty}
\frac{(2n+1)^{d}(n+1)-(2n-1)^{d}(n-1)}{n^{d+\alpha-1}}\biggr) |\theta|^{2}
&\text{for~} \alpha>2,
\end{dcases}
\end{equation}
where $|\theta|:=\sqrt{\theta_{1}^{2}+\cdots+\theta_{d}^{2}}$ is the
Euclidean norm of the vector $\theta=\{\theta_{1},\ldots,\theta_{d}\}$.
\end{theorem}

The function $A_{d}(\theta)$ is positive for $\theta\neq0$ and homogeneous of
order $\alpha$, that is, $A_{d}(t\theta)\equiv t^{\alpha}A_{d}(\theta)$ for
$t\ge 0$. From here the next corollary follows.

\begin{corollary}Let $d\ge2$ and $0<\alpha<2$. Then
\[
F_{d}(\theta)\simeq
\frac{2}{\alpha}\Gamma(1-\alpha)\cos\Bigl(\frac{\pi\alpha}{2}\Bigr)
|\theta|^{\alpha}A_{d}\Bigl(\frac{\theta}{|\theta|}\Bigr),
\]
where $0<c\le A_{d}(\theta)\le C<\infty$ for all $\theta$ satisfying
$|\theta|=1$.\end{corollary}

The integral representation \eqref{E:defAd} for the function $A_{d}(\theta)$
in \eqref{E:asympA} may be found inconvenient for one reason or another. It
is possible to get rid of integrals in \eqref{E:defAd}, and thereby in
\eqref{E:asympA}, by taking advantage of the fact that the multiple integral
in \eqref{E:defAd} can be computed explicitly.
\begin{theorem}\label{T:3}
Let $d\ge2$ and $0<\alpha<2$. Then the function $A_{d}(\theta)$ admits the
following alternative representation:
\begin{multline}\label{E:asympSA}
A_{d}(\theta)=\frac{1}{(\alpha+1)\cdots(\alpha+d-1)}\times\\\times\sum_{\{j_{k}\}\in\setJ_{1}}\frac{1}{\theta_{j_{2}}\cdots\theta_{j_{d}}}
\sum_{s_{2},\ldots,s_{d}=\pm1}s_{2}\cdots s_{d}
(\theta_{j_{1}}+ s_{2}\theta_{j_{2}}+\cdots+ s_{d}\theta_{j_{d}})^{(d-1,\alpha)},
\end{multline}
where $x^{(m,\alpha)}:=x^{m}|x|^{\alpha}$.
\end{theorem}

\begin{example} Let $0<\alpha<2$ then
\begin{align*}
F_{2}(\theta)&\simeq
\frac{2}{\alpha (\alpha+1)}\Gamma(1-\alpha)\cos\left(\frac{\pi\alpha}{2}\right)\\
&\begin{multlined}[b][0.8\linewidth]\times\Bigl(\frac{1}{\theta_{2}}\bigl((\theta_{1}+\theta_{2})|\theta_{1}+\theta_{2}|^{\alpha}
-(\theta_{1}-\theta_{2})|\theta_{1}-\theta_{2}|^{\alpha}\bigr)\\
+\frac{1}{\theta_{1}}\bigl((\theta_{2}+\theta_{1})|\theta_{2}+\theta_{1}|^{\alpha}
-(\theta_{2}-\theta_{1})|\theta_{2}-\theta_{1}|^{\alpha}\bigr)\Bigr).
\end{multlined}
\end{align*}
\end{example}

Now consider the function
\begin{equation}\label{E:defF}
F(\theta)=\sum_{z\in\bZ^{d}\setminus\{0\}}a_{z}(1-\cos\langle z,\theta\rangle),
\qquad \theta\in\bR^{d},
\end{equation}
defined by a series of more general form than \eqref{E:defFd}. If the
coefficients $a_{z}$ satisfy
\begin{equation}\label{E:boundcond}
0<a_{*}\le a_{z}\|z\|^{d+\alpha}\le a^{*}<\infty,\quad z\in\mathbb{Z}^{d}\setminus\{0\},
\end{equation}
then
\begin{equation}\label{E:equiv}
a_{*}F_{d}(\theta)\le F(\theta)\le a^{*}F_{d}(\theta),
\end{equation}
that is, in this case, one can only say that, for every $\alpha>0$, the order
of decrease of the function $F(\theta)$ at zero is the same as the order of
decrease of the function $F_{d}(\theta)$. Under additional assumptions, the
asymptotic behavior of the function $F(\theta) $ can be specified explicitly,
if to use Theorem~\ref{T:2} for $F_{d}(\theta)$.

\begin{theorem}\label{T:genas}Let $0<\alpha\le 2$ and
$a_{z}\|z\|^{d+\alpha}\to 1$  as $\|z\|\to\infty$. Then $F(\theta)\simeq
F_{d}(\theta)$ as $\theta\to 0$.
\end{theorem}

In the case when $\alpha>2$ and $a_{z}\|z\|^{d+\alpha}\to 1$ for
$\|z\|\to\infty $, as in the case of \eqref{E:boundcond}, we again can only
argue that $F(\theta)$ will have the same order of decrease at zero as the
function $F_{d}(\theta)$.

\section{Proof of Theorem~\ref{T:1}}\label{S:proofT1}

We will begin with transformation of the function $F_{d}(\theta)$ to a form
playing a key role in the subsequent constructions:
\[
F_{d}(\theta)=\sum_{z\in\bZ^{d}}\frac{1}{\|z\|^{d+\alpha}}\bigr(1-e^{i\langle z,\theta\rangle}\bigl)
=\sum_{n=1}^{\infty}\frac{1}{n^{d+\alpha}}\sum_{\|z\|=n}\bigr(1-e^{i\langle z,\theta\rangle}\bigl),
\]
and then
\begin{align*}
F_{d}(\theta)&=\sum_{n=1}^{\infty}\frac{1}{n^{d+\alpha}}\biggl(\sum_{\|z\|\le n}\bigr(1-e^{i\langle z,\theta\rangle}\bigl)-
\sum_{\|z\|\le n-1}\bigr(1-e^{i\langle z,\theta\rangle}\bigl)\biggr)\\
&=\sum_{n=1}^{\infty}\frac{1}{n^{d+\alpha}}\biggl((2n+1)^{d}-(2n-1)^{d}-
\sum_{\|z\|\le n}e^{i\langle z,\theta\rangle}+
\sum_{\|z\|\le n-1}e^{i\langle z,\theta\rangle}\biggr)\\
&=\sum_{n=1}^{\infty}\frac{1}{n^{d+\alpha}}\bigl((2n+1)^{d}-(2n-1)^{d}\bigr)\\
&\begin{multlined}[b][0.8\linewidth]\quad-\sum_{n=1}^{\infty}\frac{1}{n^{d+\alpha}}\biggl(\sum_{|k_{1}|,\ldots,|k_{d}|\le n}e^{ik_{1}\theta_{1}}\cdots e^{ik_{d}\theta_{d}}\\
-\sum_{|k_{1}|,\ldots,|k_{d}|\le n-1}e^{ik_{1}\theta_{1}}\cdots e^{ik_{d}\theta_{d}}\biggr).
\end{multlined}
\end{align*}

Let us note that
\[
\sum_{k=-n}^{n}e^{ik\theta}=
\frac{e^{i(n+\frac{1}{2})\theta}-e^{-i(n+\frac{1}{2})\theta}}{e^{i\frac{\theta}{2}}-e^{-i\frac{\theta}{2}}}=
\frac{\sin(n+\frac{1}{2})\theta}{\sin\frac{\theta}{2}}.
\]
for any real $\theta\neq0$. Therefore the function $F_{d}(\theta)$ can be
represented in the form
\begin{multline}\label{E:Fdview}
F_{d}(\theta)=\sum_{n=1}^{\infty}\frac{1}{n^{d+\alpha}}\biggl((2n+1)^{d}-(2n-1)^{d}\\
-\biggl(\prod_{j=1}^{d}\frac{\sin(n+\frac{1}{2})\theta_{j}}{\sin\frac{\theta_{j}}{2}}-
\prod_{j=1}^{d}\frac{\sin(n-\frac{1}{2})\theta_{j}}{\sin\frac{\theta_{j}}{2}}\biggr)\biggr).
\end{multline}
So, we have proved the following
\begin{lemma}\label{L:Fdview}
The function $F_{d}(\theta)$ can be represented in the form
\begin{equation}\label{E:fd}
F_{d}(\theta)=\sum_{n=1}^{\infty}\frac{1}{n^{d+\alpha}}f_{d,n}(\theta),
\end{equation}
where
\begin{equation}\label{E:fd-term}
f_{d,n}(\theta)=(2n+1)^{d}-(2n-1)^{d}-
\frac{1}{\prod_{j=1}^{d}\sin\frac{\theta_{j}}{2}}S_{d}\Bigl(n\theta,\frac{\theta}{2}\Bigr)
\end{equation}
and
\[
S_{d}(\varphi,\psi):=\prod_{j=1}^{d}\sin(\varphi_{j}+\psi_{j})-
\prod_{j=1}^{d}\sin(\varphi_{j}-\psi_{j}),\qquad \varphi,\psi\in\bR^{d}.
\]
\end{lemma}

Lemma~\ref{L:Fdview} indicates that properties of the function
$F_{d}(\theta)$ are determined by properties of the function
$S_{d}(\varphi,\psi)$. Therefore let us investigate the function
$S_{d}(\varphi,\psi)$ in more detail.

\begin{lemma}\label{L:sumcos}
The function $S_{d}(\varphi,\psi)$ can be represented in the form
\begin{multline*}
S_{d}(\varphi,\psi)=2\sum_{m\in\setO_{d}}\sum_{\{j_{k}\}\in\setJ_{m}} \sin\psi_{j_{1}}\cdots\sin\psi_{j_{m}}\cos\psi_{j_{m+1}}\cdots\cos\psi_{j_{d}}\\
\times
\cos\varphi_{j_{1}}\cdots\cos\varphi_{j_{m}}\sin\varphi_{j_{m+1}}\cdots\sin\varphi_{j_{d}}.
\end{multline*}
\end{lemma}

\begin{proof}
By definition
\[
S_{d}(\varphi,\psi)=S_{d}^{+}(\varphi,\psi)-S_{d}^{-}(\varphi,\psi),
\]
where
\begin{align*}
S_{d}^{+}(\varphi,\psi)&=\prod_{j=1}^{d}\sin(\varphi_{j}+\psi_{j})=
\prod_{j=1}^{d}(\sin\varphi_{j}\cos\psi_{j}+\cos\varphi_{j}\sin\psi_{j}),\\
S_{d}^{-}(\varphi,\psi)&=\prod_{j=1}^{d}\sin(\varphi_{j}-\psi_{j})=
\prod_{j=1}^{d}(\sin\varphi_{j}\cos\psi_{j}-\cos\varphi_{j}\sin\psi_{j}).
\end{align*}
From here it is seen that the function $S_{d}^{+}(\varphi,\psi)$, as well as
the function $S_{d}^{+}(\varphi,\psi)$, is a sum of products of sines and
cosines in the variables $\varphi_{j}$ and $\psi_{j}$. The only difference
between them is that in the representation of $S_{d}^{+}(\varphi,\psi)$ all
the products of sines and cosines are prepended by the plus sign while in the
representation of $S_{d}^{-}(\varphi,\psi)$ half of such products are
prepended by the minus sign. Under summation half of such products of sines and cosines will be
mutually eliminated while others will be doubled. More specifically,
there will be doubled those products in the expansion of
$S_{d}^{-}(\varphi,\psi)$ that have odd number of multipliers of the form
$\cos\varphi_{j}\sin\psi_{j}$, which is reflected in the condition of the
lemma. The lemma is proved.
\end{proof}

Lemma~\ref{L:sumcos} shows that, in the representation of the function
$S_{d}(\varphi,\psi)$ as a sum of products of sines and cosines, each product
consists of odd number of cosine multipliers in the variables $\varphi_{j}$
and some number of sine multipliers in the variables $\varphi_{j}$. In
connection with this let us find out how do the products of cosines in
different variables or the products of sines  preceded by one cosine look
like.

\begin{lemma}\label{L-cosprod}
For each $m\ge1$ the following representation is valid:
\begin{equation}\label{E:cosprod}
\cos\varphi_{1}\cos\varphi_{2}\cdots\cos\varphi_{m}=
\frac{1}{2^{m}}\sum_{\{s_{k}\}\in\setU_{m}}
\cos(s_{1}\varphi_{1}+\cdots+s_{m}\varphi_{m}).
\end{equation}
\end{lemma}

\begin{proof}
For $m=1$ equality \eqref{E:cosprod} is obvious. Further the proof is carried
out by induction.
\end{proof}

\begin{lemma}\label{L-sinprod}
For each $d\ge2$ the following representation is valid:
\[
\cos\varphi_{1}\sin\varphi_{2}\cdots\sin\varphi_{d}=
\frac{1}{2^{d-1}}\int\limits_{-\varphi_{d}}^{\varphi_{d}}\cdots
\int\limits_{\mathclap{-\varphi_{2}}}^{\mathclap{\varphi_{2}}}
\cos(\varphi_{1}+\sigma_{2}+\cdots+\sigma_{d})\,d\sigma_{2}\ldots d\sigma_{d}.
\]
\end{lemma}

\begin{proof}
For $d=2$ the claim is verified directly. Further the proof is carried out by
induction.
\end{proof}

If the product
$\cos\varphi_{1}\cdots\cos\varphi_{m}\sin\varphi_{m+1}\cdots\sin\varphi_{d}$
contains not less than two cosines, that is $m\ge2$, then the following lemma
holds which generalizes Lemma~\ref{L-sinprod}.
\begin{lemma}\label{L-cossinprod}
For $1\le m< d$, $d\ge 2$, the following representation is valid:
\begin{multline*}
\cos\varphi_{1}\cdots\cos\varphi_{m}\sin\varphi_{m+1}\cdots\sin\varphi_{d}\\
=\frac{1}{2^{d}}\sum_{\{s_{k}\}\in\setU_{m}}\int\limits_{-\varphi_{d}}^{\varphi_{d}}\cdots
\int\limits_{\mathclap{-\varphi_{m+1}}}^{\mathclap{\varphi_{m+1}}}
\cos(s_{1}\varphi_{1}+\cdots+s_{m}\varphi_{m}+\sigma_{m+1}+\cdots+\sigma_{d})\,d\sigma_{m+1}\ldots d\sigma_{d}.
\end{multline*}
\end{lemma}

\begin{proof}
Expanding by Lemma~\ref{L-cosprod} the product of cosines
$\cos\varphi_{1}\cdots\cos\varphi_{m}$ in a sum of cosines we obtain:
\begin{multline*}
\cos\varphi_{1}\cdots\cos\varphi_{m}\sin\varphi_{m+1}\cdots\sin\varphi_{d}\\
=\frac{1}{2^{m}}\sum_{\{s_{k}\}\in\setU_{m}}
\cos(s_{1}\varphi_{1}+\cdots+s_{m}\varphi_{m})\sin\varphi_{m+1}\cdots\sin\varphi_{d}.
\end{multline*}
Then, by applying to the summands in the right-hand part the integral
representation from Lemma~\ref{L-sinprod}, we deduce the claim of the lemma.
\end{proof}

In view of Lemma~\ref{L:Fdview} the arguments of the function
$S_{d}(\varphi,\psi)$ in the representation \eqref{E:fd}, \eqref{E:fd-term}
of $F_{d}(\theta)$ have a special form: $\varphi= n\theta$ and
$\psi=\frac{\theta}{2}$. Therefore to complete the proof of Theorem~\ref{T:1}
we need to provide a more detailed analysis of properties of the function
$S_{d}\Bigl(n\theta,\frac{\theta}{2}\Bigr)$.

By Lemma~\ref{L:sumcos}
\begin{multline*}
\frac{1}{\prod_{j=1}^{d}\sin\frac{\theta_{j}}{2}}S_{d}\Bigl(n\theta,\frac{\theta}{2}\Bigr)\\
=2\sum_{m\in\setO_{d}}\sum_{\{j_{k}\}\in\setJ_{m}}
\cot\frac{\theta_{j_{m+1}}}{2}\cdots\cot\frac{\theta_{j_{d}}}{2}
\cos n\theta_{j_{1}}\cdots\cos n\theta_{j_{m}}\sin n\theta_{j_{m+1}}\cdots\sin n\theta_{j_{d}}.
\end{multline*}
Replacing by Lemma~\ref{L-cossinprod} in this last equality the products of
sines and cosines in the variables $n\theta_{j}$ by the related sums of
integrals we get:
\begin{multline*}
\frac{1}{\prod_{j=1}^{d}\sin\frac{\theta_{j}}{2}}S_{d}\Bigl(n\theta,\frac{\theta}{2}\Bigr)\\
\quad\quad=2\frac{1}{2^{d}}\sum_{m\in\setO_{d}}\sum_{\{j_{k}\}\in\setJ_{m}} \biggl(\cot\frac{\theta_{j_{m+1}}}{2}\cdots\cot\frac{\theta_{j_{d}}}{2}\\
\quad\quad\quad\times\sum_{\{s_{k}\}\in\setU_{m}}\int\limits_{-n\theta_{j_{d}}}^{n\theta_{j_{d}}}\cdots
\int\limits_{\mathclap{-n\theta_{j_{m+1}}}}^{\mathclap{n\theta_{j_{m+1}}}}
\cos(s_{1}n\theta_{j_{1}}+\cdots+
s_{m}n\theta_{j_{m}}\\
+\sigma_{m+1}+\cdots+\sigma_{d})\,d\sigma_{m+1}\ldots d\sigma_{d}\biggr).
\end{multline*}
Carrying out the change of variables $\sigma_{j}=n\eta_{j}$ in these
integrals we achieve that cosines under the integrals will depend on $n$-fold
arguments:
\begin{multline*}
\frac{1}{\prod_{j=1}^{d}\sin\frac{\theta_{j}}{2}}S_{d}\Bigl(n\theta,\frac{\theta}{2}\Bigr)\\
\quad\quad=2\frac{1}{2^{d}}\sum_{m\in\setO_{d}}\sum_{\{j_{k}\}\in\setJ_{m}} n^{d-m}\biggl(\cot\frac{\theta_{j_{m+1}}}{2}\cdots\cot\frac{\theta_{j_{d}}}{2}\\
\quad\quad\quad\times\sum_{\{s_{k}\}\in\setU_{m}}\int\limits_{-\theta_{j_{d}}}^{\theta_{j_{d}}}\cdots
\int\limits_{\mathclap{-\theta_{j_{m+1}}}}^{\mathclap{\theta_{j_{m+1}}}}
\cos n(s_{1}\theta_{j_{1}}+\cdots+
s_{m}\theta_{j_{m}}\\
+ \eta_{m+1}+\cdots+ \eta_{d})\,d \eta_{m+1}\ldots d \eta_{d}\biggr).
\end{multline*}
Hence the term $f_{d,n}(\theta)$ in \eqref{E:fd}, \eqref{E:fd-term} can be
represented in the following form
\begin{align*}
f_{d,n}(\theta)&=(2n+1)^{d}-(2n-1)^{d}\\
&-2\frac{1}{2^{d}}\sum_{m\in\setO_{d}}\sum_{\{j_{k}\}\in\setJ_{m}} n^{d-m}\biggl(\cot\frac{\theta_{j_{m+1}}}{2}\cdots\cot\frac{\theta_{j_{d}}}{2}\\
&\begin{multlined}[b][0.8\linewidth]\quad\times\sum_{\{s_{k}\}\in\setU_{m}}\int\limits_{-\theta_{j_{d}}}^{\theta_{j_{d}}}\cdots
\int\limits_{\mathclap{-\theta_{j_{m+1}}}}^{\mathclap{\theta_{j_{m+1}}}}
\cos n(s_{1}\theta_{j_{1}}+\cdots+
s_{m}\theta_{j_{m}}\\
+ \eta_{m+1}+\cdots+ \eta_{d})\,d \eta_{m+1}\ldots d \eta_{d}\biggr).
\end{multlined}
\end{align*}
From here, by subtracting unity from cosine under the integral and then again
adding unity to it, we obtain
\begin{align*}
f_{d,n}(\theta)&=(2n+1)^{d}-(2n-1)^{d}\\
&\begin{multlined}[b][0.8\linewidth]-2\frac{1}{2^{d}}\sum_{m\in\setO_{d}}\sum_{\{j_{k}\}\in\setJ_{m}} n^{d-m}
\biggl(\cot\frac{\theta_{j_{m+1}}}{2}\cdots\cot\frac{\theta_{j_{d}}}{2}\\
\times
\sum_{\{s_{k}\}\in\setU_{m}}\int\limits_{-\theta_{j_{d}}}^{\theta_{j_{d}}}\cdots
\int\limits_{\mathclap{-\theta_{j_{m+1}}}}^{\mathclap{\theta_{j_{m+1}}}}
1\,d \eta_{m+1}\ldots d \eta_{d}\biggr)\end{multlined}
\\
&+2\frac{1}{2^{d}}\sum_{m\in\setO_{d}}\sum_{\{j_{k}\}\in\setJ_{m}} n^{d-m}\biggl(\cot\frac{\theta_{j_{m+1}}}{2}\cdots\cot\frac{\theta_{j_{d}}}{2}\\
&\begin{multlined}[b][0.8\linewidth]\quad\times\sum_{\{s_{k}\}\in\setU_{m}}\int\limits_{-\theta_{j_{d}}}^{\theta_{j_{d}}}\cdots
\int\limits_{\mathclap{-\theta_{j_{m+1}}}}^{\mathclap{\theta_{j_{m+1}}}}
\bigl(1-\cos n(s_{1}\theta_{j_{1}}+\cdots+
s_{m}\theta_{j_{m}}\\+ \eta_{m+1}+\cdots+ \eta_{d})\bigr)\,d \eta_{m+1}\ldots d \eta_{d}\biggr).
\end{multlined}
\end{align*}
Here
\begin{multline*}
\cot\frac{\theta_{j_{m+1}}}{2}\cdots\cot\frac{\theta_{j_{d}}}{2}
\sum_{\{s_{k}\}\in\setU_{m}}\int\limits_{-\theta_{j_{d}}}^{\theta_{j_{d}}}\cdots
\int\limits_{\mathclap{-\theta_{j_{m+1}}}}^{\mathclap{\theta_{j_{m+1}}}}
1\,d \eta_{m+1}\ldots d \eta_{d}\\
\quad\quad=\cot\frac{\theta_{j_{m+1}}}{2}\cdots\cot\frac{\theta_{j_{d}}}{2}\cdot 2^{m}\cdot
(2\theta_{m+1}\cdots 2\theta_{d})\\
=2^{d}\Bigl(\theta_{j_{m+1}}\cot\frac{\theta_{j_{m+1}}}{2}\Bigr)
\cdots\Bigl(\theta_{j_{d}}\cot\frac{\theta_{j_{d}}}{2}\Bigr).
\end{multline*}
Then
\begin{align*}
f_{d,n}(\theta)&=(2n+1)^{d}-(2n-1)^{d}\\
&-2\sum_{m\in\setO_{d}}\sum_{\{j_{k}\}\in\setJ_{m}} n^{d-m}\Bigl(\theta_{j_{m+1}}\cot\frac{\theta_{j_{m+1}}}{2}\Bigr)
\cdots\Bigl(\theta_{j_{d}}\cot\frac{\theta_{j_{d}}}{2}\Bigr)
\\
&+2\frac{1}{2^{d}}\sum_{m\in\setO_{d}}\sum_{\{j_{k}\}\in\setJ_{m}} n^{d-m}\biggl(\cot\frac{\theta_{j_{m+1}}}{2}\cdots\cot\frac{\theta_{j_{d}}}{2}\\
&\begin{multlined}[b][0.8\linewidth]\quad\times\sum_{\{s_{k}\}\in\setU_{m}}\int\limits_{-\theta_{j_{d}}}^{\theta_{j_{d}}}\cdots
\int\limits_{\mathclap{-\theta_{j_{m+1}}}}^{\mathclap{\theta_{j_{m+1}}}}
\bigl(1-\cos n(s_{1}\theta_{j_{1}}+\cdots+
s_{m}\theta_{j_{m}}\\
+ \eta_{m+1}+\cdots+ \eta_{d})\bigr)\,d \eta_{m+1}\ldots d \eta_{d}\biggr)
\end{multlined}
\end{align*}
or, what is the same,
\begin{align*}
f_{d,n}(\theta)&=(2n+1)^{d}-(2n-1)^{d}
-2\sum_{m\in\setO_{d}}\sum_{\{j_{k}\}\in\setJ_{m}} (2n)^{d-m}\\
&+2\sum_{m\in\setO_{d}}\sum_{\{j_{k}\}\in\setJ_{m}} n^{d-m}\Bigl(2^{d-m}-\Bigl(\theta_{j_{m+1}}\cot\frac{\theta_{j_{m+1}}}{2}\Bigr)
\cdots\Bigl(\theta_{j_{d}}\cot\frac{\theta_{j_{d}}}{2}\Bigr)\Bigr)
\\
&+2\frac{1}{2^{d}}\sum_{m\in\setO_{d}}\sum_{\{j_{k}\}\in\setJ_{m}} n^{d-m}\biggl(\cot\frac{\theta_{j_{m+1}}}{2}\cdots\cot\frac{\theta_{j_{d}}}{2}\\
&\begin{multlined}[b][0.8\linewidth]\quad\times\sum_{\{s_{k}\}\in\setU_{m}}\int\limits_{-\theta_{j_{d}}}^{\theta_{j_{d}}}\cdots
\int\limits_{\mathclap{-\theta_{j_{m+1}}}}^{\mathclap{\theta_{j_{m+1}}}}
\bigl(1-\cos n(s_{1}\theta_{j_{1}}+\cdots+
s_{m}\theta_{j_{m}}\\+ \eta_{m+1}+\cdots+ \eta_{d})\bigr)\,d \eta_{m+1}\ldots d \eta_{d}\biggr).
\end{multlined}
\end{align*}
Here the external sums are taken over all the sets of indices $\{j_{k}\}$
with odd number of members $j_{1},\ldots,j_{m}$. Then
\[
2  \sum_{m\in\setO_{d}}\sum_{\{j_{k}\}\in\setJ_{m}} (2n)^{d-m}=2\sum_{m=1,3,\ldots\le d}
\binom{d}{m}(2n)^{d-m}=(2n+1)^{d}-(2n-1)^{d},
\]
and therefore
\begin{multline}\label{E:fdn-fin}
f_{d,n}(\theta)=
2\sum_{m\in\setO_{d}}\sum_{\{j_{k}\}\in\setJ_{m}} n^{d-m}\Bigl(2^{d-m}-\Bigl(\theta_{j_{m+1}}\cot\frac{\theta_{j_{m+1}}}{2}\Bigr)
\cdots\Bigl(\theta_{j_{d}}\cot\frac{\theta_{j_{d}}}{2}\Bigr)\Bigr)
\\
\hphantom{f_{d,n}(\theta)}+2\frac{1}{2^{d}}\sum_{m\in\setO_{d}}\sum_{\{j_{k}\}\in\setJ_{m}} n^{d-m}\biggl(\cot\frac{\theta_{j_{m+1}}}{2}\cdots\cot\frac{\theta_{j_{d}}}{2}\\
\hphantom{f_{d,n}(\theta)}\quad\times\sum_{\{s_{k}\}\in\setU_{m}}\int\limits_{-\theta_{j_{d}}}^{\theta_{j_{d}}}\cdots
\int\limits_{\mathclap{-\theta_{j_{m+1}}}}^{\mathclap{\theta_{j_{m+1}}}}
\bigl(1-\cos n(s_{1}\theta_{j_{1}}+\cdots+
s_{m}\theta_{j_{m}}\\
+ \eta_{m+1}+\cdots+ \eta_{d})\bigr)\,d \eta_{m+1}\ldots d \eta_{d}\biggr).
\end{multline}

Substituting now the obtained expression in \eqref{E:fd} and changing there
the order of summation, using the fact that all the sums in \eqref{E:fdn-fin}
are finite, we obtain
\begin{align*}
F_{d}(\theta)&=\sum_{n=1}^{\infty}\frac{1}{n^{d+\alpha}}f_{d,n}(\theta)\\
&=2\sum_{m\in\setO_{d}}\sum_{\{j_{k}\}\in\setJ_{m}}\biggl(\sum_{n=1}^{\infty}\frac{n^{d-m}}{n^{d+\alpha}}\biggr)
\Bigl(2^{d-m}-\Bigl(\theta_{j_{m+1}}\cot\frac{\theta_{j_{m+1}}}{2}\Bigr)
\cdots\Bigl(\theta_{j_{d}}\cot\frac{\theta_{j_{d}}}{2}\Bigr)\Bigr)\\
&+2\frac{1}{2^{d}}\sum_{m\in\setO_{d}}\sum_{\{j_{k}\}\in\setJ_{m}} \biggl(\cot\frac{\theta_{j_{m+1}}}{2}\cdots\cot\frac{\theta_{j_{d}}}{2}\\
&\begin{multlined}[b][0.8\linewidth]\quad\times\sum_{\{s_{k}\}\in\setU_{m}}\int\limits_{-\theta_{j_{d}}}^{\theta_{j_{d}}}\cdots
\int\limits_{\mathclap{-\theta_{j_{m+1}}}}^{\mathclap{\theta_{j_{m+1}}}}
\Bigl(\sum_{n=1}^{\infty}\frac{n^{d-m}}{n^{d+\alpha}}\bigl(1-\cos n(s_{1}\theta_{j_{1}}+\cdots+
s_{m}\theta_{j_{m}}\\
+ \eta_{m+1}+\cdots+ \eta_{d})\bigl)\Bigr)\,d \eta_{m+1}\ldots d \eta_{d}\biggr),
\end{multlined}
\end{align*}
from which representation \eqref{E:Fd-fin} follows. Theorem~\ref{T:1} is
proved.

\section{Proof of Theorem~\ref{T:2}}\label{S:proofT2}

Let us note that the function $F_{d}(\theta)$ is even and it does not change
values under any permutation of the coordinates of the vector
$\theta=\{\theta_{1},\theta_{2},\ldots,\theta_{d}\}$, and also under any
change of sign of this vector. Therefore, in what follows in this section,
without loss in generality the coordinates of the vector $\theta$ can be
taken as nonnegative.

First show that the asymptotic behavior of the function $F_{d}(\theta)$ is
defined actually only by those summands in \eqref{E:Fd-fin} which relate to
the case $m=1$.

\begin{lemma}\label{L:asympFd}
For $d\ge2$ and $0<\alpha\le2$ the following equality is valid:
\begin{multline}\label{E:asymp}
F_{d}(\theta)=
2\sum_{\{j_{k}\}\in\setJ_{1}} \frac{1}{\theta_{j_{2}}\cdots\theta_{j_{d}}}
\int\limits_{-\theta_{j_{d}}}^{\theta_{j_{d}}}\cdots
\int\limits_{\mathclap{-\theta_{j_{2}}}}^{\mathclap{\theta_{j_{2}}}}
H_{\alpha}(\theta_{j_{1}}+\eta_{2}+\cdots+ \eta_{d})\,d \eta_{2}\ldots d \eta_{d}
\\+O(\|\theta\|^{2}).
\end{multline}
\end{lemma}

\begin{proof}
Let us estimate the rate of growth of different summands in \eqref{E:Fd-fin}.

{\emph{First sum.}} We start with the estimation of the first external sum in
\eqref{E:Fd-fin}:
\[
F_{d}^{(1)}(\theta)=
2\sum_{m\in\setO_{d}}\sum_{\{j_{k}\}\in\setJ_{m}}\zeta(m+\alpha)
\Bigl(2^{d-m}-\Bigl(\theta_{j_{m+1}}\cot\frac{\theta_{j_{m+1}}}{2}\Bigr)
\cdots\Bigl(\theta_{j_{d}}\cot\frac{\theta_{j_{d}}}{2}\Bigr)\Bigr).
\]
From this inequality it is seen that $F_{d}^{(1)}(\theta)$ is a finite sum of
the terms
\[
2^{d-m}-\Bigl(\theta_{j_{m+1}}\cot\frac{\theta_{j_{m+1}}}{2}\Bigr)
\cdots\Bigl(\theta_{j_{d}}\cot\frac{\theta_{j_{d}}}{2}\Bigr)
\]
multiplied by some numerical factors. Since
$x\cot\frac{x}{2}=2-\frac{1}{6}x^{2}+\ldots$, where the dots denote the terms
of higher order of smallness, then
\begin{equation}\label{E:ctrprod}
0\le 2^{d-m}-\Bigl(\theta_{j_{m+1}}\cot\frac{\theta_{j_{m+1}}}{2}\Bigr)
\cdots\Bigl(\theta_{j_{d}}\cot\frac{\theta_{j_{d}}}{2}\Bigr)
\lesssim \theta_{j_{m+1}}^{2}+\cdots+\theta_{j_{d}}^{2}\lesssim \|\theta\|^{2}
\end{equation}
for sufficiently small values of $|\theta|$. Here for two functions
$h_{1}(x)$ and $h_{2}(x)$ we write $h_{1}(x)\lesssim h_{2}(x)$ if there
exists a constant $C<\infty$ such that $h_{1}(x)\le C h_{2}(x)$ in a
neighborhood of the point $x=0$; we also write $h_{1}(x)\gtrsim h_{2}(x)$ if
there exists a constant $c>0$ such that $h_{1}(x)\ge c h_{2}(x)$ in a
neighborhood of the zero point. Then \eqref{E:ctrprod} implies
\begin{equation}\label{E:Fd-fin1-ineq}
\Bigl|F_{d}^{(1)}(\theta)\Bigr|\lesssim \|\theta\|^{2}.
\end{equation}

\emph{Second sum.} Consider the second external sum in \eqref{E:Fd-fin}:
\begin{align*}
F_{d}^{(2)}(\theta)&=
\frac{1}{2^{d-1}}\sum_{m\in\setO_{d}}\sum_{\{j_{k}\}\in\setJ_{m}} \biggl(\cot\frac{\theta_{j_{m+1}}}{2}\cdots\cot\frac{\theta_{j_{d}}}{2}\\
&\begin{multlined}[b][0.8\linewidth]\quad\times\sum_{\{s_{k}\}\in\setU_{m}}\int\limits_{-\theta_{j_{d}}}^{\theta_{j_{d}}}\cdots
\int\limits_{\mathclap{-\theta_{j_{m+1}}}}^{\mathclap{\theta_{j_{m+1}}}}
H_{m+\alpha-1}(s_{1}\theta_{j_{1}}+\cdots+
s_{m}\theta_{j_{m}}\\+ \eta_{m+1}+\cdots+ \eta_{d})\,d \eta_{m+1}\ldots d \eta_{d}\biggr).
\end{multlined}
\end{align*}
Clearly $F_{d}^{(2)}(\theta)$ is a finite sum of the terms
\begin{multline}\label{E:Rm}
R_{m,j_{1},\ldots,j_{d},s_{1},\ldots,s_{m}}(\theta):=
\frac{1}{2^{d-1}}\cot\frac{\theta_{j_{m+1}}}{2}\cdots\cot\frac{\theta_{j_{d}}}{2}\\
\quad\times\int\limits_{-\theta_{j_{d}}}^{\theta_{j_{d}}}\cdots
\int\limits_{\mathclap{-\theta_{j_{m+1}}}}^{\mathclap{\theta_{j_{m+1}}}}
H_{m+\alpha-1}(s_{1}\theta_{j_{1}}+\cdots+
s_{m}\theta_{j_{m}}\\+ \eta_{m+1}+\cdots+ \eta_{d})\,d \eta_{m+1}\ldots d \eta_{d},
\end{multline}
where $\{s_{k}\}\in\setU_{m}$.

Analysis of behavior of the terms
$R_{m,j_{1},\ldots,j_{d},s_{1},\ldots,s_{m}}(\theta)$ will be different
depending on whether $m=1$ or $m=3,5,\ldots$. Let us start with the second of
these cases.

\emph{Case $m=3,5,\ldots$\,.} In this case $m+\alpha-1>2$ and then by
\eqref{E:asympF1} $0\le H_{m+\alpha-1}(x)\lesssim x^{2}$. Therefore
\begin{multline*}
0\le H_{m+\alpha-1}(s_{1}\theta_{j_{1}}+\cdots+
s_{m}\theta_{j_{m}}+ \eta_{m+1}+\cdots+ \eta_{d})\\
\lesssim
\bigl(|\theta_{j_{1}}|+\cdots+
|\theta_{j_{m}}|+ |\theta_{j_{m+1}}|+\cdots+ |\theta_{j_{d}}|\bigr)^{2}
\lesssim \|\theta\|^{2}
\end{multline*}
for all $\eta_{m+1}\in[-\theta_{j_{m+1}},\theta_{j_{m+1}}],\ldots,
\eta_{d}\in[-\theta_{j_{d}},\theta_{j_{d}}]$. Then for all sufficiently small
values of $\|\theta\|$ we have
\begin{multline*}
0\le R_{m,j_{1},\ldots,j_{d},s_{1},\ldots,s_{m}}(\theta)\lesssim
\cot\frac{\theta_{j_{m+1}}}{2}\cdots\cot\frac{\theta_{j_{d}}}{2}
\int\limits_{-\theta_{j_{d}}}^{\theta_{j_{d}}}\cdots
\int\limits_{\mathclap{-\theta_{j_{m+1}}}}^{\mathclap{\theta_{j_{m+1}}}}
\|\theta\|^{2}\,d \eta_{m+1}\ldots d \eta_{d}\\
=\Bigl(2\theta_{j_{m+1}}\cot\frac{\theta_{j_{m+1}}}{2}\Bigr)
\cdots\Bigl(2\theta_{j_{d}}\cot\frac{\theta_{j_{d}}}{2}\Bigr)\|\theta\|^{2}
\lesssim\|\theta\|^{2},
\end{multline*}
since each factor $2\theta_{j_{k}}\cot\frac{\theta_{j_{k}}}{2}$ is bounded
for all sufficiently small values of $\theta_{j_{k}}$. Hence
\begin{equation}\label{E:Rm3}
0\le R_{m,j_{1},\ldots,j_{d},s_{1},\ldots,s_{m}}(\theta)\lesssim
\|\theta\|^{2}.
\end{equation}

\emph{Case $m=1$.} By verbatim repetition of the estimates from the previous
case with usage of the representation \eqref{E:asympF1} we obtain, for
$0<\alpha\le 2$, that
\begin{equation}\label{E:Rm1up}
0\le R_{1,j_{1},\ldots,j_{d},s_{1}}(\theta)\lesssim
H^{*}_{\alpha}(\|\theta\|).
\end{equation}
Let us show that the inverse estimate is also valid:
\begin{equation}\label{E:Rm1low}
R_{1,j_{1},\ldots,j_{d},s_{1}}(\theta)\gtrsim
H^{*}_{\alpha}(\|\theta\|).
\end{equation}

By \eqref{E:asympF1} there exists a $q>0$ such that $H_{\alpha}(x)\ge q
H^{*}_{\alpha}(x)$ for all sufficiently small values of $x\ge0$. Then, taking
into account that by assumption all the coordinates of the vector $\theta$
are nonnegative, we obtain
\[
R_{1,j_{1},\ldots,j_{d},s_{1}}(\theta)
\ge q
\cot\frac{\theta_{j_{2}}}{2}\cdots\cot\frac{\theta_{j_{d}}}{2}\int\limits_{-\theta_{j_{d}}}^{\theta_{j_{d}}}\cdots
\int\limits_{\mathclap{-\theta_{j_{2}}}}^{\mathclap{\theta_{j_{2}}}}
H^{*}_{\alpha}(s_{1}\theta_{j_{1}}+\eta_{2}+\cdots+ \eta_{d})\,d \eta_{2}\ldots d \eta_{d},
\]
Since here the function $H^{*}_{\alpha}(x)$ is even then the right-hand side
does not depend on the sign of $s_{1}=\pm1$, and therefore
\[
R_{1,j_{1},\ldots,j_{d},s_{1}}(\theta)
\ge q
\cot\frac{\theta_{j_{2}}}{2}\cdots\cot\frac{\theta_{j_{d}}}{2}\int\limits_{-\theta_{j_{d}}}^{\theta_{j_{d}}}\cdots
\int\limits_{\mathclap{-\theta_{j_{2}}}}^{\mathclap{\theta_{j_{2}}}}
H^{*}_{\alpha}(\theta_{j_{1}}+\eta_{2}+\cdots+ \eta_{d})\,d \eta_{2}\ldots d \eta_{d},
\]
Since the integrand here is nonnegative then, by appropriate reducing the
region of integration, we can obtain the following inequality:
\[
R_{1,j_{1},\ldots,j_{d},s_{1}}(\theta)
\ge q
\cot\frac{\theta_{j_{2}}}{2}\cdots\cot\frac{\theta_{j_{d}}}{2}\int\limits_{\theta_{j_{d}}/2}^{\theta_{j_{d}}}\cdots
\int\limits_{\mathclap{\theta_{j_{2}}/2}}^{\mathclap{\theta_{j_{2}}}}
H^{*}_{\alpha}(\theta_{j_{1}}+\eta_{2}+\cdots+ \eta_{d})\,d \eta_{2}\ldots d \eta_{d},
\]
from which, in view of monotone non-decreasing of the function
$H^{*}_{\alpha}(x)$ for small $x\ge0$, it follows that
\[
R_{1,j_{1},\ldots,j_{d},s_{1}}(\theta)
\ge q
\cot\frac{\theta_{j_{2}}}{2}\cdots\cot\frac{\theta_{j_{d}}}{2}\int\limits_{\theta_{j_{d}}/2}^{\theta_{j_{d}}}\cdots
\int\limits_{\mathclap{\theta_{j_{2}}/2}}^{\mathclap{\theta_{j_{2}}}}
H^{*}_{\alpha}\Bigl(\theta_{j_{1}}+\frac{\theta_{j_{2}}}{2}+\cdots+ \frac{\theta_{j_{d}}}{2}\Bigr)\,d \eta_{2}\ldots d \eta_{d},
\]
because $\eta_{k}\ge \frac{\theta_{j_{k}}}{2}$, $k=2,\ldots,d$, for
$\eta_{k}\in\bigl[\frac{\theta_{j_{k}}}{2},\theta_{j_{k}}\bigr]$. Then
\begin{multline*}
R_{1,j_{1},\ldots,j_{d},s_{1}}(\theta)
\ge q
\Bigl(\frac{\theta_{j_{2}}}{2}\cot\frac{\theta_{j_{2}}}{2}\Bigr)
\cdots
\Bigl(\frac{\theta_{j_{d}}}{2}\cot\frac{\theta_{j_{d}}}{2}\Bigr)
H^{*}_{\alpha}\Bigl(\theta_{j_{1}}+\frac{\theta_{j_{2}}}{2}+\cdots+ \frac{\theta_{j_{d}}}{2}\Bigr)\\
\ge
q\Bigl(\frac{\theta_{j_{2}}}{2}\cot\frac{\theta_{j_{2}}}{2}\Bigr)
\cdots
\Bigl(\frac{\theta_{j_{d}}}{2}\cot\frac{\theta_{j_{d}}}{2}\Bigr)
H^{*}_{\alpha}\Bigl(\frac{\|\theta\|}{2}\Bigr),
\end{multline*}
which, for small values of $\|\theta\|$, implies~\eqref{E:Rm1low}.

Relations~\eqref{E:Rm1up} and \eqref{E:Rm1low} can be united in the following
one:
\begin{equation}\label{E:Rm1fin}
R_{1,j_{1},\ldots,j_{d},s_{1}}(\theta)\sim
H^{*}_{\alpha}(\|\theta\|).
\end{equation}

In conclusion let us simplify the expression for
$R_{1,j_{1},\ldots,j_{d},s_{1}}(\theta)$ by multiplying and then dividing the
right-hand part of \eqref{E:Rm} on
$\frac{\theta_{j_{m+1}}}{2}\cdots\frac{\theta_{j_{d}}}{2}$:
\begin{equation}\label{E:rm-via-Rm*}
R_{1,j_{1},\ldots,j_{d},s_{1}}(\theta)=
\Bigl(\frac{\theta_{j_{m+1}}}{2}\cot\frac{\theta_{j_{m+1}}}{2}\Bigr)
\cdots\Bigl(\frac{\theta_{j_{d}}}{2}\cot\frac{\theta_{j_{d}}}{2}\Bigr)R^{*}_{1,j_{1},\ldots,j_{d},s_{1}}(\theta)
\end{equation}
where
\begin{multline}\label{E:Rm*}
R^{*}_{m,j_{1},\ldots,j_{d},s_{1},\ldots,s_{m}}(\theta):=
\frac{1}{\theta_{j_{m+1}}\cdots\theta_{j_{d}}}\int\limits_{-\theta_{j_{d}}}^{\theta_{j_{d}}}\cdots
\int\limits_{\mathclap{-\theta_{j_{m+1}}}}^{\mathclap{\theta_{j_{m+1}}}}
H_{m+\alpha-1}(s_{1}\theta_{j_{1}}+\cdots+
s_{m}\theta_{j_{m}}\\
+ \eta_{m+1}+\cdots+ \eta_{d})\,d \eta_{m+1}\ldots d \eta_{d},
\end{multline}

By \eqref{E:ctrprod}
\[
\Bigl(\frac{\theta_{j_{m+1}}}{2}\cot\frac{\theta_{j_{m+1}}}{2}\Bigr)
\cdots\Bigl(\frac{\theta_{j_{d}}}{2}\cot\frac{\theta_{j_{d}}}{2}\Bigr)
=1+O(\|\theta\|^{2}),
\]
where $O(\|\theta\|^{2})$ is a term whose order of smallness is not less than
two. Then \eqref{E:Rm1fin} implies
\begin{equation}\label{E:Rm*fin}
R^{*}_{1,j_{1},\ldots,j_{d},s_{1}}(\theta)\sim
H^{*}_{\alpha}(\|\theta\|)
\end{equation}
and
\begin{equation}\label{E:rm-via-Rm*-fin}
R_{1,j_{1},\ldots,j_{d},s_{1}}(\theta)=R^{*}_{1,j_{1},\ldots,j_{d},s_{1}}(\theta)
+O(\|\theta\|^{2}).
\end{equation}
Remark that by the definition \eqref{E:asympF1} the quantity
$H^{*}_{\alpha}(\|\theta\|)$, for $0<\alpha<2$ and small values of
$\|\theta\|$, has the order of growth as $\|\theta\|^{\alpha}$ while for
$\alpha=2$ its order of growth is $\|\theta\|^{2}\log\frac{1}{\|\theta\|}$.

\emph{Final asymptotics.} Relations~\eqref{E:Fd-fin1-ineq}, \eqref{E:Rm3},
\eqref{E:Rm1fin}, \eqref{E:Rm*}, \eqref{E:Rm*fin} and
\eqref{E:rm-via-Rm*-fin} show that in the case $0<\alpha\le 2$ the
asymptotics of the function $F_{d}(\theta)$ as $\theta\to 0$ is determined by
the terms $R^{*}_{1,j_{1},\ldots,j_{d},s_{1}}(\theta)$. All the other terms
in~\eqref{E:Fd-fin} have higher order of smallness at the origin which is not
less than two. Therefore
\begin{multline*}
F_{d}(\theta)=
\sum_{\{j_{k}\}\in\setJ_{1}} \biggl(\frac{1}{\theta_{j_{2}}\cdots\theta_{j_{d}}}\\
\times\sum_{s_{1}=\pm1}\int\limits_{-\theta_{j_{d}}}^{\theta_{j_{d}}}\cdots
\int\limits_{\mathclap{-\theta_{j_{2}}}}^{\mathclap{\theta_{j_{2}}}}
H_{\alpha}(s_{1}\theta_{j_{1}}
+\eta_{2}+\cdots+ \eta_{d})\,d \eta_{2}\ldots d \eta_{d}\biggr)
+O(\|\theta\|^{2}).
\end{multline*}
At last, taking into account an earlier remark that the case $s_{1}=-1$ is
reduced to the case $s_{1}=1$, we obtain (by doubling the number of summands
corresponding to the case $s_{1}=1$) representation~\eqref{E:asymp}.
Lemma~\ref{L:asympFd} is proved.
\end{proof}

Let us complete the proof of theorem.

\emph{Case $0<\alpha<2$.} To obtain the required asymptotic representation
\eqref{E:asympA} in this case it suffices to substitute the asymptotic
expression \eqref{E:asympF1} for the function $H_{\alpha}(\cdot)$ in equality
\eqref{E:asymp} from Lemma~\ref{L:asympFd}. However one should bear in mind
that formula \eqref{E:asympF1} does not provide any estimates for smallness
of the ``reminder terms'' in the asymptotic representation of the integrand
function $H_{\alpha}(\cdot)$. Owing to this we cannot affirm that the
reminder terms in \eqref{E:asympA} are of the order $O(|\theta|^{2})$ and
would have to prove a weaker asymptotic equality.

\emph{Case $\alpha=2$.} As in the previous case, by substituting the
asymptotic expression \eqref{E:asympF1} for the function $H_{\alpha}(\cdot)$
with $\alpha=2$ in equality \eqref{E:asymp} from Lemma~\ref{L:asympFd} we
obtain:
\begin{equation}\label{E:Fd-lapha2}
F_{d}(\theta)\simeq
\sum_{\{j_{k}\}\in\setJ_{1}} \biggl(\frac{1}{\theta_{j_{2}}\cdots\theta_{j_{d}}}
\int\limits_{-\theta_{j_{d}}}^{\theta_{j_{d}}}\cdots
\int\limits_{\mathclap{-\theta_{j_{2}}}}^{\mathclap{\theta_{j_{2}}}}
L(\theta_{j_{1}}+\eta_{2}+\cdots+ \eta_{d})\,d \eta_{2}\ldots d \eta_{d}\biggr),
\end{equation}
where $L(x):=x^{2}\log\Bigl(\frac{1}{|x|}\Bigr)$ for small $x\in\mathbb{R}$.
By using the directly verifiable equality
\[
L(x)=y^{2}L\Bigl(\frac{x}{y}\Bigr)+x^{2}\log\Bigl(\frac{1}{y}\Bigr),
\quad x,y\in\mathbb{R},~ y\neq0,
\]
represent the integrand term in \eqref{E:Fd-lapha2} in the following form:
\begin{multline*}
L(\theta_{j_{1}}+\eta_{2}+\cdots+ \eta_{d})\\=
|\theta|^{2}L\Bigl(\frac{\theta_{j_{1}}+\eta_{2}+\cdots+ \eta_{d}}{|\theta|}\Bigr)
+
(\theta_{j_{1}}+\eta_{2}+\cdots+ \eta_{d})^{2}\log\Bigl(\frac{1}{|\theta|}\Bigr),
\end{multline*}
where $|\theta|$ is the Euclidean norm of the vector $\theta$. Here
$|\theta_{j_{1}}+\eta_{2}+\cdots+ \eta_{d}|\le d|\theta|$ for all
$\theta_{j_{1}}$, $|\eta_{2}|\le \theta_{j_{2}},\ldots, |\eta_{d}|\le
\theta_{j_{d}}$, and therefore by continuity of the function $L(x)$
\[
\left|L\Bigl(\frac{\theta_{j_{1}}+\eta_{2}+\cdots+ \eta_{d}}{|\theta|}\Bigr)\right|\le
L^{*}
\]
for some $L^{*}<\infty$.

Hence equality \eqref{E:Fd-lapha2} is reduced to the form
\begin{multline*}
F_{d}(\theta)\simeq
\log\Bigl(\frac{1}{|\theta|}\Bigr)\sum_{\{j_{k}\}\in\setJ_{1}} \biggl(\frac{1}{\theta_{j_{2}}\cdots\theta_{j_{d}}}\\
\times
\int\limits_{-\theta_{j_{d}}}^{\theta_{j_{d}}}\cdots
\int\limits_{\mathclap{-\theta_{j_{2}}}}^{\mathclap{\theta_{j_{2}}}}
(\theta_{j_{1}}+\eta_{2}+\cdots+ \eta_{d})^{2}\,d \eta_{2}\ldots d \eta_{d}\biggr)+O(|\theta|^{2}),
\end{multline*}
and we needed only to calculate the integrals in this last representation. By
expanding each integrand we obtain that
\[
(\theta_{j_{1}}+\eta_{2}+\cdots+ \eta_{d})^{2}=
\theta_{j_{1}}^{2}+\eta_{2}^{2}+\cdots+ \eta_{d}^{2}+\ldots,
\]
where the dots denote the mixed products of variables
$\theta_{j_{1}},\eta_{2},\ldots,\eta_{d}$. Since each such mixed product is
an odd function in its variables then integrating of such a function gives a
zero contribution to the integral. From here
\begin{multline*}
F_{d}(\theta)\simeq
\log\Bigl(\frac{1}{|\theta|}\Bigr)\sum_{\{j_{k}\}\in\setJ_{1}} \biggl(\frac{1}{\theta_{j_{2}}\cdots\theta_{j_{d}}}\\
\times\int\limits_{-\theta_{j_{d}}}^{\theta_{j_{d}}}\cdots
\int\limits_{\mathclap{-\theta_{j_{2}}}}^{\mathclap{\theta_{j_{2}}}}
(\theta_{j_{1}}^{2}+\eta_{2}^{2}+\cdots+ \eta_{d}^{2})\,d \eta_{2}\ldots d \eta_{d}\biggr)+O(|\theta|^{2}),
\end{multline*}
Now the obtained integrals can be computed directly:
\begin{align*}
F_{d}(\theta)&\simeq
\log\Bigl(\frac{1}{|\theta|}\Bigr)\sum_{\{j_{k}\}\in\setJ_{1}}2^{d-1}
\Bigl(\theta_{j_{1}}^{2}+\frac{1}{3}\theta_{j_{2}}^{2}+\cdots+ \frac{1}{3}\theta_{j_{d}}^{2}\Bigr)+O(|\theta|^{2})\\
&\begin{multlined}[b][0.8\linewidth]=
2^{d-1}\Bigl(1+\frac{1}{3}(d-1)\Bigr)(\theta_{1}^{2}+\theta_{2}^{2}+\cdots+\theta_{d}^{2})\log\Bigl(\frac{1}{|\theta|}\Bigr)+O(|\theta|^{2})\\
\simeq \frac{2^{d-1}(d+2)}{3}|\theta|^{2}\log\Bigl(\frac{1}{|\theta|}\Bigr).
\end{multlined}
\end{align*}
The theorem is proved in the case $\alpha=2$.

\emph{Case $\alpha>2$.} In this case the function $F_{d}(\theta)$ is twice
continuously differentiable and has a minimum at the point $\theta=0$.
Therefore its asymptotics is determined by terms of the second or higher
order of smallness in its Taylor expansion. Unfortunately it is not too easy
to find the required coefficients of the Taylor series directly from formula
\eqref{E:defFd}. Therefore we will make use of representation
\eqref{E:Fdview} for $F_{d}(\theta)$ or, what is the same, of representation
\eqref{E:fd}, \eqref{E:fd-term} from Lemma~\ref{L:Fdview}. By
\eqref{E:Fdview}, to compute the quadratic terms of the function
$F_{d}(\theta)$ it suffices to find for each $n=1,2,\ldots$ the quadratic
terms of the functions
\[
\tilde{S}_{n}(\theta)=\prod_{j=1}^{d}\frac{\sin(n+\frac{1}{2})\theta_{j}}{\sin\frac{\theta_{j}}{2}}-
\prod_{j=1}^{d}\frac{\sin(n-\frac{1}{2})\theta_{j}}{\sin\frac{\theta_{j}}{2}}
\]

By using the following expansions
\begin{align*}
\frac{\sin(n+\frac{1}{2})x}{{\sin\frac{x}{2}}}&=
(2n+1)\Bigl(1-\frac{n^{2}+n}{6}x^{2}\Bigr)+\dotsb,\\
\frac{\sin(n-\frac{1}{2})x}{{\sin\frac{x}{2}}}&=
(2n-1)\Bigl(1-\frac{n^{2}-n}{6}x^{2}\Bigr)+\dotsb,
\end{align*}
where the dots stand for the terms of higher than second order of smallness,
we get:
\begin{align*}
\tilde{S}_{n}(\theta)&=(2n+1)^{d}-(2n-1)^{d}\\
&\quad-\frac{1}{6}\bigl((2n+1)^{d}(n^{2}+n)-(2n-1)^{d}(n^{2}-n)\bigr)
(\theta_{1}^{2}+\cdots+\theta_{d}^{2})+\dotsb\\
&\begin{multlined}[b][0.8\linewidth]=(2n+1)^{d}-(2n-1)^{d}\\
-\frac{1}{6}\bigl((2n+1)^{d}(n^{2}+n)-(2n-1)^{d}(n^{2}-n)\bigr)
|\theta|^{2}+\dotsb .
\end{multlined}
\end{align*}

Substituting the obtained expression in \eqref{E:Fdview} and making there
summation over the terms up to the second order of smallness inclusive we
obtain the required representation \eqref{E:asympA}. Theorem~\ref{T:2} is
proved.

\section{Proof of Theorem~\ref{T:3}}\label{S:proofT4}

Let us note that for each integer $m\ge1$ and real $\alpha>0$ the following
equality holds:
\begin{equation}\label{E:diff}
\frac{d}{dx}(x^{m}|x|^{\alpha})= (\alpha+m)x^{m-1}|x|^{\alpha}.
\end{equation}
To prove this equality it suffices to represent the factor $|x|^{\alpha}$ as
$|x|^{\alpha}=(x^{2})^{\frac{\alpha}{2}}$ and then make termwise
differentiation of the obtained expression routinely. Then
\[
\frac{d^{k}}{dx^{k}}(x^{m}|x|^{\alpha})=
(\alpha+m-k+1)\cdots(\alpha+m)x^{m-k}|x|^{\alpha},
\]
for $k=1,2,\ldots,m$ and, in particular,
\[
\frac{d^{m}}{dx^{m}}(x^{m}|x|^{\alpha})=
(\alpha+1)\cdots(\alpha+m)|x|^{\alpha}.
\]
Therefore
\begin{multline*}
\int\limits_{-x_{m}}^{x_{m}}\cdots
\int\limits_{\mathclap{-x_{2}}}^{\mathclap{x_{2}}}
|x_{1}+ y_{2}+\cdots+ y_{m}|^{\alpha}\,d y_{2}\ldots d y_{m}\\
=\frac{1}{(\alpha+1)\cdots(\alpha+m)}\sum_{s_{2},\ldots,s_{m}=\pm1}s_{2}\cdots s_{m}(x_{1}+ s_{2}x_{2}+\cdots+ s_{m}x_{m})^{(m,\alpha)},
\end{multline*}
from which \eqref{E:asympSA} follows. Theorem~\ref{T:3} is proved.

\section{Proof of Theorem~\ref{T:genas}}\label{S-proofG}

Given arbitrary $\varepsilon>0$, by the condition of Theorem~\ref{T:genas}
such a $\rho=\rho(\varepsilon)$ can be found that
$\Bigl|1-a_{z}\|z\|^{d+\alpha}\Bigr|\le\varepsilon$ for $\|z\|\ge\rho$. Then
\begin{multline*}
1-\frac{F(\theta)}{F_{d}(\theta)}=\frac{F_{d}(\theta)-F(\theta)}{F_{d}(\theta)}=
\frac{1}{F_{d}(\theta)}\sum_{z:\|z\|<\rho}
\Bigl(\frac{1}{\|z\|^{d+\alpha}}-a_{z}\Bigr)(1-\cos\langle z,\theta\rangle)\\+
\frac{1}{F_{d}(\theta)}\sum_{z:\|z\|\ge\rho}
\Bigl(\frac{1}{\|z\|^{d+\alpha}}-a_{z}\Bigr)(1-\cos\langle z,\theta\rangle)
\end{multline*}
Here the first summand in the right-hand part tends to zero as $\theta\to 0$
since by Theorem~\ref{T:1} $F_{d}(\theta)\sim \|\theta\|^{\alpha}$ in the
case $0<\alpha<2$ and $F_{d}(\theta)\sim
\log(\frac{1}{\|\theta\|})\|\theta\|^{2}$ in the case $\alpha=2$ for small
values of $\|\theta\|$, whereas all the terms under the summation sign, the
number of which is finite, have quadratic order of smallness at zero.

The second summand does not exceed $\varepsilon$ by absolute value since by
the choice of $\rho$ we have:
\begin{multline*}
\left|\frac{1}{F_{d}(\theta)}\sum_{z:\|z\|\ge\rho}
\Bigl(\frac{1}{\|z\|^{d+\alpha}}-a_{z}\Bigr)(1-\cos\langle z,\theta\rangle)\right|\\
\quad\le
\frac{1}{F_{d}(\theta)}\sum_{z:\|z\|\ge\rho}
\Bigl|1-a_{z}\|z\|^{d+\alpha}\Bigr|\frac{1-\cos\langle z,\theta\rangle}{\|\theta\|^{d+\alpha}}\\
\le
\varepsilon\frac{1}{F_{d}(\theta)}\sum_{z:\|z\|\ge\rho}
\frac{1-\cos\langle z,\theta\rangle}{\|\theta\|^{d+\alpha}}
\le\varepsilon\frac{F_{d}(\theta)}{F_{d}(\theta)}
\le\varepsilon.
\end{multline*}
Hence
\[
\limsup_{\theta\to 0}\left|1-\frac{F(\theta)}{F_{d}(\theta)}\right|\le\varepsilon,
\]
and in view of arbitrariness of $\varepsilon$ we obtain that $F(\theta)\simeq
F_{d}(\theta)$. Theorem~\ref{T:genas} is proved.

\section*{Acknowledgments}
The research was carried out at the Institute for Information Transmission
Problems, Russian Academy of Science, at the expense of the Russian
Foundation for Sciences (project ¹ 14-50-00150).

The author is grateful to E. Yarovaya, who drew his attention to the
importance of calculating the Hardy type asymptotics of multi-variable
trigonometric series in probability theory problems and thereby stimulated
interest in the issues under discussion.


\begin{thebibliography}{10}
\expandafter\ifx\csname url\endcsname\relax
  \def\url#1{\texttt{#1}}\fi
\expandafter\ifx\csname urlprefix\endcsname\relax\def\urlprefix{URL }\fi
\expandafter\ifx\csname href\endcsname\relax
  \def\href#1#2{#2} \def\path#1{#1}\fi

\bibitem{Hardy:LJMS28} G.~H. Hardy,
    \href{http://jlms.oxfordjournals.org/content/s1-3/1/12}{A theorem
  concerning trigonometrical series}, J. London Math. Soc. S1-3~(1) (1928)
  12--13.
\newblock \href {http://dx.doi.org/10.1112/jlms/s1-3.1.12}
  {\path{doi:10.1112/jlms/s1-3.1.12}}.
\newline\urlprefix\url{http://jlms.oxfordjournals.org/content/s1-3/1/12}

\bibitem{Hardy:PLMS31} G.~H. Hardy,
    \href{http://plms.oxfordjournals.org/content/s2-32/1/441}{Some
  theorems concerning trigonometrical series of a special type}, Proc. London
  Math. Soc. S2-32~(1) (1931) 441--448.
\newblock \href {http://dx.doi.org/10.1112/plms/s2-32.1.441}
  {\path{doi:10.1112/plms/s2-32.1.441}}.
\newline\urlprefix\url{http://plms.oxfordjournals.org/content/s2-32/1/441}

\bibitem{ABT:PIMA56} S.~Aljan{\v{c}}i{\'{c}}, R.~Bojani{\'{c}},
    M.~Tomi{\'{c}}, Sur le comportement
  asymtotique au voisinage de z{\'{e}}ro des s{\'{e}}ries
  trigonom{\'{e}}triques de sinus {\`{a}} coefficients monotones, Acad. Serbe
  Sci., Publ. Inst. Math. 10 (1956) 101--120.

\bibitem{Yong:JMAA71} C.-H. Yong,
  \href{http://www.sciencedirect.com/science/article/pii/0022247X71901788}{On
  the asymptotic behavior of trigonometric series. {I}}, J. Math. Anal. Appl.
  33 (1971) 23--34.
\newblock \href {http://dx.doi.org/10.1016/0022-247X(71)90178-8}
  {\path{doi:10.1016/0022-247X(71)90178-8}}.
\newline\urlprefix\url{http://www.sciencedirect.com/science/article/pii/0022247X71901788}

\bibitem{Yong:JMAA72} C.-H. Yong,
  \href{http://www.sciencedirect.com/science/article/pii/0022247X72901114}{On
  the asymptotic behavior of trigonometric series. {II}}, J. Math. Anal. Appl.
  38 (1972) 1--14.
\newblock \href {http://dx.doi.org/10.1016/0022-247X(72)90111-4}
  {\path{doi:10.1016/0022-247X(72)90111-4}}.
\newline\urlprefix\url{http://www.sciencedirect.com/science/article/pii/0022247X72901114}

\bibitem{HR:QJM45} G.~H. Hardy, W.~W. Rogosinski, Notes on {F}ourier series.
    {III}. {A}symptotic
  formulae for the sums of certain trigonometrical series, Quart. J. Math.,
  Oxford Ser. 16 (1945) 49--58.

\bibitem{Nurcombe:JMAA93} J.~R. Nurcombe,
  \href{http://www.sciencedirect.com/science/article/pii/S0022247X83712916}{On
  trigonometric series with quasimonotone coefficients}, J. Math. Anal. Appl.
  178~(1) (1993) 63--69.
\newblock \href {http://dx.doi.org/10.1006/jmaa.1993.1291}
  {\path{doi:10.1006/jmaa.1993.1291}}.
\newline\urlprefix\url{http://www.sciencedirect.com/science/article/pii/S0022247X83712916}

\bibitem{ChenChen:JMAA00} C.-P. Chen, L.~Chen,
  \href{http://www.sciencedirect.com/science/article/pii/S0022247X0096952X}{Asymptotic
  behavior of trigonometric series with {$O$}-regularly varying quasimonotone
  coefficients}, J. Math. Anal. Appl. 250~(1) (2000) 13--26.
\newblock \href {http://dx.doi.org/10.1006/jmaa.2000.6952}
  {\path{doi:10.1006/jmaa.2000.6952}}.
\newline\urlprefix\url{http://www.sciencedirect.com/science/article/pii/S0022247X0096952X}

\bibitem{ChenChen:JMAA00-II} C.-P. Chen, L.~Chen,
  \href{http://www.sciencedirect.com/science/article/pii/S0022247X00967453}{Asymptotic
  behavior of trigonometric series with {$O$}-regularly varying quasimonotone
  coefficients. {II}}, J. Math. Anal. Appl. 245~(1) (2000) 297--301.
\newblock \href {http://dx.doi.org/10.1006/jmaa.2000.6745}
  {\path{doi:10.1006/jmaa.2000.6745}}.
\newline\urlprefix\url{http://www.sciencedirect.com/science/article/pii/S0022247X00967453}

\bibitem{Tikhonov:JMAA07} S.~Tikhonov,
  \href{http://www.sciencedirect.com/science/article/pii/S0022247X0600179X}{Trigonometric
  series with general monotone coefficients}, J. Math. Anal. Appl. 326~(1)
  (2007) 721--735.
\newblock \href {http://dx.doi.org/10.1016/j.jmaa.2006.02.053}
  {\path{doi:10.1016/j.jmaa.2006.02.053}}.
\newline\urlprefix\url{http://www.sciencedirect.com/science/article/pii/S0022247X0600179X}

\bibitem{Zygmund02} A.~Zygmund, Trigonometric series. {V}ol. {I}, {II}, 3rd
    Edition, Cambridge
  Mathematical Library, Cambridge University Press, Cambridge, 2002, with a
  foreword by Robert A. Fefferman.

\bibitem{Yar:CSTM13} E.~Yarovaya,
  \href{http://www.tandfonline.com/doi/abs/10.1080/03610926.2012.703282}{Branching
  random walks with heavy tails}, Comm. Statist. Theory Methods 42~(16) (2013)
  3001--3010.
\newblock \href {http://dx.doi.org/10.1080/03610926.2012.703282}
  {\path{doi:10.1080/03610926.2012.703282}}.
\newline\urlprefix\url{http://www.tandfonline.com/doi/abs/10.1080/03610926.2012.703282}

\bibitem{Yar:NPSMDA14} E.~B. Yarovaya, Criteria for transient behavior of
    symmetric branching random
  walks on {$\mathbf{Z}$} and {$\mathbf{Z}^2$}, in: J.~R. Bozeman, V.~Girardin,
  C.~H. Skiadas (Eds.), New Perspectives on Stochastic Modeling and Data
  Analysis, ISAST Athens Greece, 2014, pp. 283--294.

\end{thebibliography}

  \providecommand{\bbljan}[0]{January}
  \providecommand{\bblfeb}[0]{February}
  \providecommand{\bblmar}[0]{March}
  \providecommand{\bblapr}[0]{April}
  \providecommand{\bblmay}[0]{May}
  \providecommand{\bbljun}[0]{June}
  \providecommand{\bbljul}[0]{July}
  \providecommand{\bblaug}[0]{August}
  \providecommand{\bblsep}[0]{September}
  \providecommand{\bbloct}[0]{October}
  \providecommand{\bblnov}[0]{November}
  \providecommand{\bbldec}[0]{December}

\end{document}